\documentclass[18pt]{article}
\usepackage{amsfonts}
\usepackage{amsmath} 
\usepackage{amssymb}
\usepackage{amsthm}
\usepackage{tikz}
\usetikzlibrary{shapes.geometric}
\usepackage[margin=1in]{geometry}

\newtheorem{theorem}{Theorem}[section]
\newtheorem{lemma}[theorem]{Lemma}
\newtheorem{question}[theorem]{Question}
\newtheorem{proposition}[theorem]{Proposition}
\newtheorem{corollary}[theorem]{Corollary}
\newtheorem{definition}[theorem]{Definition}

\newcommand{\C}{\mathbb{C}}

\newcommand{\bth}{\begin{theorem}}
\renewcommand{\eth}{\end{theorem}}
\newcommand{\bpr}{\begin{proposition}}
\newcommand{\epr}{\end{proposition}}
\newcommand{\bde}{\begin{definition}}
\newcommand{\ede}{\end{definition}}
\newcommand{\blem}{\begin{lemma}}
\newcommand{\elem}{\end{lemma}}
\newcommand{\bco}{\begin{corollary}}
\newcommand{\eco}{\end{corollary}}
\newcommand{\prove}{\begin{proof}}
\newcommand{\done}{\end{proof}}

\newcommand{\ite}{\begin{itemize}}
\newcommand{\mize}{\end{itemize}}

\newcommand{\ben}{\begin{enumerate}}
\newcommand{\een}{\end{enumerate}}

\newcommand{\Q}{\mathbb{Q}}

\newcommand{\Z}{\mathbb{Z}}

\DeclareMathOperator{\Res}{Res}
\DeclareMathOperator{\sign}{sign}
\DeclareMathOperator{\supp}{supp}

\DeclareMathOperator{\Gal}{Gal}
\title{Cyclotomic Points and Algebraic Properties of Polygon Diagonals}
\author{Thomas Grubb\\\texttt{tgrubb@ucsd.edu}\\University of California, San Diego \and Christian Woll\\\texttt{christianwoll.1@gmail.com}\\University of California, San Diego}
\date{\today\\[10pt]
	\begin{flushleft}
	                                       \\[5pt]
	\end{flushleft}}
\begin{document}

\maketitle
\abstract{By viewing the regular $N$-gon as the set of $N$th roots of unity in the complex plane we transform several questions regarding polygon diagonals into when a polynomial vanishes when evaluated at roots of unity. To study these solutions we implement algorithms in Sage as well as examine a trigonometric diophantine equation. In doing so we classify when a metallic ratio can be realized as a ratio of polygon diagonals, answering a question raised in a PBS Infinite Series broadcast. We then generalize this idea by examining the degree of the number field generated by a given ratio of polygon diagonals. }
\section{Introduction}
The study of rational solutions to (possibly several) polynomial equations has a long history and needs no introduction. Lately there has been interest in extending this line of study to other types of algebraic numbers. In particular one may ask the following. Given $k$ multivariate polynomials $f_1,\dots,f_k$ from $\mathbb{Q}[x_1,\dots,x_n]$, does there exist a point $(\zeta_1,\dots,\zeta_n)$ with both 
\begin{itemize}
    \item $\zeta_j$ a root of unity for all $j$, and
    \item $f_i(\zeta_1,\dots,\zeta_n)=0$ for all $i$?
\end{itemize}
Such a solution will be called a \emph{cyclotomic point} or \emph{cyclotomic solution} to the equations defined by the $f_i$. Lang had conjectured that such solutions can be given in terms of finitely many parametric families \cite{Lan61}. This has since been verified via work of Ihara, Serre, and Tate \cite{Lan61} and Laurent \cite{Lau84}. Recent work has been centered around developing methods for computing the parametric families of cyclotomic solutions to a given set of equations. In doing so one finds applications to properties of regular polgyons \cite{PR98} and to solving trigonometric diophantine equations \cite{CJ76}.

The easiest paradigm of this question involves finding zeroes of univariate polynomials which are roots of unity; this was carried out by Bradford and Davenport \cite{BD88}, who gave a concise algorithm for finding such factors. This has since been extended by Beukers and Smyth to the setting of algebraic plane curves \cite{BS02}, and generalized to arbitrary dimensions by Aliev and Smith \cite{AS12}. These papers will form the backbone of our results; those interested in further constructive results and explicit methods may additionally explore \cite{Ler12}, \cite{Roj07}, and \cite{Rup93}. 

Our main contribution is to apply existing algorithms to explicitly compute the cyclotomic solutions to two specific polynomials. In doing so we are able to solve two classification problems. 

The first of our main results regards metallic means, which are extensions of the well known golden ratio $\phi_1=\frac{1+\sqrt{5}}{2}$. The golden ratio is ubiquitous in combinatorics and other areas, making appearances with respect to the Fibonacci sequence, irrationality and Diophantine approximation \cite{Hur1894}, and even topological entropy \cite{DDM99}. Naturally the golden ratio has been extended to \emph{metallic means}; for an integer $n\geq 1$ define the $n$th metallic mean by 
$$
\phi_n=\frac{n+\sqrt{n^2+4}}{2}.
$$
The formula comes from the generalized Fibonacci recursion 
$$
F^{(n)}_k=nF^{(n)}_{k-1}+F^{(n)}_{k-2}
$$
with appropriate initial conditions; see \cite{GW19} or \cite{Spi99} for analysis of metallic means. 

It is well known that both the golden ratio $\phi_1$ and the \emph{silver ratio} $\phi_2$ can be expressed as a ratio of polygon diagonals:
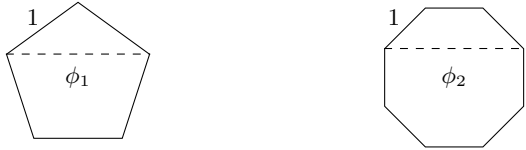
\begin{center}
\begin{tikzpicture}
\node[draw = black, minimum size=2cm,regular polygon,regular polygon sides=5] (a) at (-5,0){};
\node[draw = black, minimum size=2cm,regular polygon,regular polygon sides=8] (b) {};
\draw[dashed] (a.corner 2) to (a.corner 5);
\draw[dashed] (b.corner 3) to (b.corner 8);
\node at (-5.6,.8) {1};
\node at (-5,0) {$\phi_1$};
\node at (0,0) {$\phi_2$};
\node at (-.8,.8) {1};

\end{tikzpicture}
\end{center}
By \emph{diagonal} we mean the distance between any two distinct vertices on a regular polygon, so side lengths are allowed. In a PBS Infinite Series broadcast, Perez-Giz raised the question of whether or not the same statement is true for $\phi_n$ for any $n\geq 3$ \cite{PBS}. This question was in fact answered negatively by Buitrago for the \emph{bronze ratio} $\phi_3$ several years previously \cite{Bui07}. The methods used were entirely numeric, and did not extend to $n\geq4$. Our first contribution is the following classification, which in fact allows an even more general definition of metallic mean.
\begin{theorem}
Let $y_0$ be an algebraic number with $y_0^2\in \mathbb{Q}$. Define $\phi_{y_0}$ as the larger root of the quadratic equation
$$
x^2-y_0x-1=0.
$$
Then $\phi_{y_0}$ can be realized as a ratio of polygon diagonals if and only if 
$$
y_0\in \{0,\pm 1,\pm\tfrac{3}{2},\pm2,\pm\sqrt{2},\pm\sqrt{5},\pm\sqrt{12},\pm\tfrac{\sqrt{2}}{2},\pm\tfrac{\sqrt{6}}{6},\pm\tfrac{2\sqrt{3}}{3},\pm\tfrac{\sqrt{12}}{12}\}.
$$
\end{theorem}
Explicit constructions are given in Section 3, and the proof framework is simple to describe. By viewing the regular $N$-gon as the set of $N$th roots of unity in the complex plane, one can translate the question of Theorem 1.1 into writing $\phi_{y_0}$ as a ratio 
$$
\phi_{y_0}=\frac{|1-\zeta_N^a|}{|1-\zeta_N^b|}
$$
for some primitive root of unity $\zeta_N$. With simple algebraic manipulations one can translate this into finding a polynomial $f(y,z_1,z_2)$ such that equation (1) holds if and only if $f(y_0,\zeta_N^a,\zeta_N^b)=0$. To finish, we use elimination theory and the algorithmic tools for finding cyclotomic points on curves mentioned above to classify when solutions of this form exist. 

After proving Theorem 1.1 we will discuss a generalization of the realization of metallic means as ratios of polygon diagonals. For a given $N$th root of unity $\zeta_N$, let $\Q[\zeta_N]^+$ denote the maximal totally real subfield of the extension $\Q[\zeta_N]/\Q$; explicitly, $\Q[\zeta_N]^+=\Q[\zeta_N+\zeta_N^{-1}]$. Given diagonals of a regular polygon $d_1 = |1-\zeta_N^a|$ and $d_2 = |1-\zeta_{N}^b|$, a simple manipulation shows that the $d_i$ lie inside the real cyclotomic field $\Q[\zeta_{4N}]^+$. Should a ratio $\tfrac{d_1}{d_2}$ equal a metallic mean, then $\tfrac{d_1}{d_2}$ will generate at most a degree four subextension of $\Q[d_1,d_2]/\Q$. Thus it is natural to ask more generally, can one classify all pairs of diagonals $d_1, d_2$ for which $\tfrac{d_1}{d_2}$ generates at most a degree $C$ extension of $\Q$, for some constant $C>0$? Towards this end, we offer the following.

\begin{theorem}
Let $d_1 = |1-\zeta_N^a|$ and $d_2=|1-\zeta_N^b|$ be diagonals of a regular $N$-gon, with $\gcd(a,b)=1$. Letting $\phi$ denote Euler's totient function, we have:
\begin{itemize}
\item If $N$ is odd, then $[\Q[\tfrac{d_1}{d_2}]:\Q] = \frac{\phi(4N)}{4}$.
\item If $N$ is even, then $[\Q[\tfrac{d_1}{d_2}]:\Q]\geq \frac{\phi(4N)}{10}$.
\end{itemize}
\end{theorem} 
A direct consequence of this is that, for any $C>0$, there are only finitely many ratios of diagonals which generate an extension of degree at most $C$. Moreover, these ratios can be listed explicitly using simple bounds on the totient function $\phi$. As an example of this, we will use Theorem 1.2 to give an alternate proof of Theorem 1.1 at the end of Section 4.

Theorem 1.2 will be proven in Section 4 as a consequence to Proposition 4.4, in which we analyze the index of $\Q[\tfrac{d_1}{d_2}]$ inside $\Q[\zeta_{4N}]^+$. This analysis is similar to the proof of Theorem 1.1 in that it reduces to solving equations in roots of unity. However, due to an increased number of parameters, it is infeasible to use the same computational proof. We get around this by reducing our original equation to a trigonometric diophantine equation, to which we can apply results of W\l odarski \cite{Wlo69} and Conway and Jones \cite{CJ76}.

The rest of this paper is outlined as follows. In Section 2 we give a brief review of the algorithmic tools developed in \cite{AS12}, \cite{BD88}, and \cite{BS02}. The theoretical background will allow us to implement an algorithm in Sage \cite{sagemath} in order to solve Theorem 1.1; the discussion of this will take place in Section 3. In Section 4 we discuss briefly the computational problems for proving Theorem 1.2 in the same manner as Theorem 1.1. We are able to work around this using work of W{\l}odarski and of Conway and Jones on diophantine trigonometric equations; thus we are able to recover Theorem 1.2 with little computing at all in the end of Section 4.

\textbf{Acknowledgements:} We would like to thank Kiran Kedlaya for his patience and many helpful comments throughout this project. The first listed author would like to gratefully acknowledge support from NSF Research Training Group grant DMS-1502651 and from NSF grant DMS-1849173. 
\section{Tools for Finding Cyclotomic Points}
In this section we will discuss algorithms for finding cyclotomic points on curves or, more generally, families of curves. We will start in the simplest case and build up the theory from there. We emphasize that much of the theory discussed here may be found in \cite{BS02}; the only minor novelty that we require is to allow for a rational parameter instead of solely looking for cyclotomic solutions.
\subsection{Preliminary Definitions and Lemmas}
Given a univariate polynomial $f(x)\in \mathbb{C}[x]$ we define the (squarefree) \emph{cyclotomic part} of $f$ to be 
$$
Cf(x)=\prod_{f(\zeta)=0}(x-\zeta)
$$
where the product runs over all roots of unity on which $f$ vanishes. 

We will need the following lemma of Beukers and Smyth which characterizes roots of unity and polynomials vanishing on roots of unity. 
\begin{lemma}[BS02, Lemma 1]
If $g(x)\in \C[x]$ with $g(0)\neq 0$ is a polynomial with the property that for every zero $\alpha$ of $g$, at least one of $\pm \alpha^2$ is a zero of $g$, then all zeroes of $g$ are roots of unity. Furthermore, if $\omega$ is a root of unity, then it is conjugate (over $\Q$) to exactly one of $-\omega$, $\omega^2$, or $-\omega^2$.
\end{lemma}
From this two important lemmas follow. 
\begin{lemma}
Suppose $g(n,x)\in \Z[n,x]$ vanishes at $(n_0,\omega)$ with $n_0\in \Q$ and $\omega$ a root of unity. Then one of 
$$
g_1(n,x):=g(n,-x),\;\;g_2(n,x):=g(n,x^2),\;\;g_3(n,x):=g(n,-x^2)
$$
also vanishes at $(n_0,\omega)$.\qed
\end{lemma}
\begin{lemma}
Suppose $g(n,x)\in \Z[x,n]$ is an irreducible polynomial which is nonconstant in both $n$ and $x$. Then $g(n,x)\nmid g(n,\pm x^2)$.
\end{lemma}
\begin{proof}
Suppose $g(n,x)$ divides $g(n,\pm x^2)$. For any fixed $n_0$, Lemma 2.1 thus implies that the roots of $g(n_0,x)$ are all roots of unity. It is known that there are only finitely many monic polynomials of a given degree whose zeroes are all on the unit circle; these are the so called \emph{Kronecker polynomials} [Dam]. Since for any $n_0$ $g(n_0,x)$ is a scalar multiple of a Kronecker polynomial we may conclude by the pigeonhole principle that there is a Kronecker polynomial $K(x)$ and infinitely many distinct constants $n_k$ for which $g(n_k,x)=c_k K(x)$ for some constants $c_k$. 

We finish by writing 
$$
g(n,x)=\sum_{j=0}^NC_j(n)x^j
$$
for some coefficient polynomials $C_j(n)$. The previous paragraph shows that the $C_j(n_k)=C_{j'}(n_k)$ for all $k$, and hence the coefficient polynomials are all identical. This gives a factorization $g(n,x)=C_N(n)K(x)$; the assumption that $g(n,x)$ is nonconstant in both $n$ and $x$ guarantees that $C_N(n)$ and $K(x)$ are nonconstant and hence we have a proper factorization of $g(n,x)$, contradicting our hypothesis.
\end{proof}
Finally we give a trivariate analog of Lemma 2.2.
\begin{lemma}
Suppose $f(n,x,y)\in \Z[n,x,y]$ vanishes at $(n_0,\omega,\tau)$ with $n_0\in \Q$ and $\omega,\tau$ roots of unity. Then one of 
\begin{itemize}
    \item $f_1(n,x,y):=f(n,-x,y)$,\\ 
    \item $f_2(n,x,y):=f(n,x,-y)$,\\ 
    \item $f_3(n,x,y):=f(n,-x,-y)$,\\ 
    \item $f_4(n,x,y):=f(n,x^2,y^2)$,\\ 
    \item $f_5(n,x,y):=f(n,-x^2,y^2)$,\\ 
    \item $f_6(n,x,y):=f(n,x^2,-y^2)$,\\ 
    \item $f_7(n,x,y):=f(n,-x^2,-y^2)$\\ 
\end{itemize}
vanishes at $(n_0,\omega,\tau)$ as well.
\end{lemma}

\begin{proof}
Pick a root of unity $\zeta$ for which $\omega=\zeta^a$ and $\tau=\zeta^b$ for coprime integers $a$ and $b$. Define $g(n,z):=f(n,z^a,z^b)$, so that $g(n_0,\zeta)=0$. By taking into account the parity of $a$ and $b$, we see that 
\begin{align*}
    g(n,-z)&\in \{f_1(n,x,y),f_2(n,x,y),f_3(n,x,y)\}\\
    g(n,z^2)&=f_4(n,x,y)\\
    g(n,-z^2)&\in \{f_5(n,x,y),f_6(n,x,y),f_7(n,x,y)\}
\end{align*}
and hence by Lemma 2.2 we are done.
\end{proof}
The last prerequisite we will need is a method to find the cyclotomic part of a polynomial $f\in \mathbb{C}[x]$. We will treat the existence of such an algorithm as a black box; for details one may read \cite{BD88} or Section 2 of \cite{BS02}.
\subsection{Finding cyclotomic points for families of polynomials}
We now describe how to find cyclotomic solutions to families of polynomials indexed by a rational parameter. That is to say, given $f(n,x)\in \Q[n,x]$ we find pairs $(n_0,\omega)$ with $n_0\in \Q$ and $\omega$ a root of unity for which $f(n_0,\omega)=0$. 

We start by recalling the \emph{resultant} of two polynomials. The resultant is a classical tool in elimination theory, and we will refer the reader to Chapter 3, Section 6 of [CLO] as opposed to going into details here. For the purposes of this paper one may take the following crude and incomplete description: given two polynomials $f(n,x),g(n,x)\in \Q[n,x]$, there exists a polynomial $\Res(f,g,x)\in \Q[n]$, read ``the resultant of $f$ and $g$ with respect to $x$,'' with the property that if $(n_0,x_0)$ is a common zero of $f$ and $g$, then $n_0$ is a zero of $\Res(f,g,x)$.
Importantly, the resultant of two polynomials may be calculated efficiently in any modern computer algebra system. Moreover, the use of resultants lowers the number of variables in question at each step, making it ideal for the classification problems we are interested in.

We now describe an algorithm based on \cite{BD88} and \cite{BS02} for finding cyclotomic points on families of polynomials.  We will maintain the notation of Lemma 2.2, i.e. if $f(n,x)\in \Q[n,x]$ then
$$
f_1(n,x)=f(n,-x),f_2(n,x)=f(n,x^2),f_3(n,x)=f(n,-x^2).
$$
\textbf{Algorithm Pseudocode}

Input: A polynomial $f(n,x)\in \Q[n,x]$.

Output: All pairs $(n_0,\omega)$ with $n\in \Q$ and $\omega$ a root of unity for which $f(n_0,\omega)=0$.
\begin{enumerate}
    \item Check if $f(n,x)$ is irreducible. If not, factor $f$ and run the algorithm on irreducible components of $f$.
    \item Check if $f(n,x)$ is a polynomial in $x^m$ for an integer $m\geq 2$. If so, define $g(n,x)$ so that $g(n,x^m)=f(n,x)$ and run the algorithm on $g(n,x)$. It is simple to translate zeroes of $g(n,x)$ into zeroes of $f(n,x)$.
    \item Check if $f(n,x)$ is constant in $n$. If so, find the cyclotomic roots of $f(1,x)$ to find zeroes $\omega_1,\dots,\omega_k$. Output $(n_0,\omega_i)$ for any rational $n_0$.
    \item Check if $f(n,x)$ is constant in $x$. If so, use the Rational Roots Theorem to find the rational roots $n_1,\dots,n_s$ of $f(n,1)$. Output the parametric family $\{(n_j, \omega)\}$ where $\omega$ is any root of unity. 
    \item For $i\in \{1,2,3\}$:
    \begin{enumerate}
        \item Compute $h_i(n)=\Res(f,f_i,x)$.
        
        \item Using the Rational Roots Theorem, find the rational roots of $h_i$. Call them $r_{i,1},\dots,r_{i,k_i}$.
        
        \item For each $r_{i,j}$, find cyclotomic roots of $f(r_{i,j},x)$. For each $\omega$ output $(r_{i,j},\omega)$.
    \end{enumerate}
\end{enumerate}
\textbf{End Algorithm}

We now discuss the validity of the code. This entails ensuring that we find \emph{all} desired solutions, as well as ensuring we do not find any extraneous solutions. Suppose $n_0\in \Q$, $\omega$ is a root of unity, and $f(n_0,\omega)=0$. By Lemma 2.2, we also have $f_i(n_0,\omega)=0$ for some $i\in \{1,2,3\}$. Thus the resultant $\Res(f,f_i,x)$ will vanish at $n_0$. We can easily find this by applying the Rational Root Theorem to $\Res(f,f_i,x)$. Once we find $n_0$ the work of \cite{BD88} and \cite{BS02} ensures we can find $\omega$ by finding the cyclotomic roots of $f(n_0,x)$. Thus $(n_0,\omega)$ will be output in Step 5 of the algorithm.

Now suppose we have found an extraneous solution. This would occur if $\Res(f,f_i,x)=0$ for some $i$. For $i=2,3$ this is ruled out by Lemma 2.3, since Step 1 has the effect of ensuring we always work with irreducible polynomials. If $\Res(f,f_1,x)=0$, then $f(n,x)$ and $f(n,-x)$ share a factor. In this case it must be that $f(n,x)$ is a rational multiple of $f(n,-x)$. This can only occur if $f(n,x)$ is a polynomial in $x^2$, or if $f(n,x)=x\cdot g(n,x)$ where $g(n,x)$ is a polynomial in $x^2$. These cases are both ruled out by Steps 1 and 2. 
\subsection{Finding cyclotomic points for families of curves}
We now bootstrap the previous algorithm to finding cyclotomic points on families of curves. Our argument in this section relies heavily upon \cite{BS02}. 

Suppose $f(n,x,y)$ is a Laurent polynomial in $x$ and $y$, i.e. $f\in \mathbb{Z}[n,x,y,x^{-1},y^{-1}]$. Write 
$$
f(n,x,y)=\sum_{a,b}c_{a,b}(n)x^ay^b.
$$
We denote the \emph{support} of $f$, $\supp(f)$, to be $\supp(f):=\{(a,b):c_{a,b}(n)\neq0\}$. For example,
$$
\supp(3nxy+n^2x^2+xy^4)=\{(1,1),(2,0),(1,4)\}.
$$
From this we may define $\mathcal{L}(f)$ to be the lattice generated by the differences of elements in $\supp(f)$. Continuing our example, 
\begin{align*}
\mathcal{L}(3xy+x^2+xy^4)&=\Z\cdot (1,-1) + \Z\cdot (0,3)+\Z\cdot (1,-4)\\
&=\Z\cdot (1,-1)\oplus\Z\cdot (0,3).
\end{align*}
The following two observations are straightforward from the definition of $\mathcal{L}$. Suppose $\mathcal{L}(f)$ is rank 1, i.e. it is generated by a single vector $(a,b)$. If we let $u=x^ay^b$, then $f(n,x,y)$ may be written as 
$$
f(n,x,y)=x^sy^t g(n,u)
$$
for some monomial $x^sy^t$ and some Laurent polynomial $g\in \mathbb{Z}[n,u,u^{-1}].$ Alternatively, suppose $\mathcal{L}(f)$ is rank $2$, generated freely by $\{(a,b),(c,d)\}$. Letting $u=x^ay^b$ and $v=x^cy^d$, one can find a Laurent polynomial $g\in \mathbb{Z}[n,u,u^{-1},v,v^{-1}]$ for which 
$\mathcal{L}(g)=\mathbb{Z}\oplus\mathbb{Z}$ and for which 
$$
f(x,y)=x^sy^tg(n,u,v)
$$
for some monomial $x^sy^t$.

We now proceed with the algorithm, maintaining the notation developed in Lemma 2.4. Our objective is as follows. Given a polynomial $f(n,x,y)\in \Q[n,x,y]$, we wish to find solutions $(n_0,\omega,\tau)$ with $n_0\in \Q$ and $\omega,\tau$ roots of unity. 

\noindent\textbf{Algorithm Pseudocode}

Input: A polynomial $f(n,x,y)\in \Q[n,x,y]$.

Output: All triples $(n_0,\omega,\tau)$ with $n\in \Q$ and $\omega,\tau$ roots of unity for which $f(n_0,\omega,\tau)=0$.

\begin{enumerate}
	\item Check if $f(n,x,y)$ is irreducible. If not, factor $f$ and run the algorithm on the irreducible factors of $f$. 
    \item Check if $\mathcal{L}(f)$ is rank $1$. If so, chose $u=x^ay^b$ so that $g(n,x,y)=x^sy^tg(n,u)$. Run Algorithm 1 on the polynomial $g(n,u)$. Translating solutions of $g(n,u)$ to roots of $f(n,x,y)$ is straightforward.
    \item Check if $\mathcal{L}(f)=\Z\oplus\Z$. If not, find $g$ for which $\mathcal{L}(g)=\mathbb{Z}\oplus\mathbb{Z}$ and $f(n,x,y)=x^sy^tg(n,u,v)$. Continue the algorithm on $g.$ For each solution obtained one can recover the solutions for $f$ by following the steps of Section 3.7 in [BS].\\
    \item For $i\in \{1,2,\dots,7\}$:
    \begin{enumerate}
        \item Compute $g_i=\Res(f,f_i,y)$.
        \item Run Algorithm $2$ on $g_i$ to obtain zeroes $(n_{i,1},\omega_{i,1}),\dots,(n_{i,j_i},\omega_{i,j_i}).$
        \begin{enumerate}
            \item For $l\in \{1,\dots,j_i\}$, compute $G_l=f(n_{i,l},\omega_{i,l},y).$
            \item Compute the cyclotomic roots of $G_l$ via Algorithm 1. For each such root $\tau,$ output $(n_{i,l},\omega_{i,l},\tau)$.
        \end{enumerate}
    \end{enumerate}
\end{enumerate}
\textbf{End Algorithm}

Again we verify that the above code returns the expected result. Suppose $f(n_0,\omega,\tau)=0$ for $n_0\in \Q$ and $\omega,\tau$ roots of unity. By Lemma 2.3, $f_i(n_0,\omega,\tau)=0$ as well for some $i\in \{1,2,\dots,7\}$. In particular the resultant $g_i$ will vanish at $(n_0,\omega)$. The correctness of Algorithm 2 ensures that $(n_0,\omega)$ will be found in Step 4, (b). Once $(n_0,\omega)$ have been identified it is then a simple matter to find $\tau$ by examining the univariate polynomial $f(n_0,\omega,y)$. 

Now suppose we have found an extraneous solution. We will be led to the same contradiction as \cite{BS02}. As with Algorithm 2, this would occur if a resultant calculation $g_i=\Res(f,f_i,y)$ resulted in $g_i(n,x)\equiv0$. Since Step 1 ensures $f$ is irreducible, this can only occur if $f$ divides $f_i$ for some $i\in \{1,2,\dots,7\}$. Suppose $f$ divides $f_1$. As with Algorithm 2, this can only occur if $f$ is a Laurent polynomial in $x^2$. But this would imply that $\mathcal{L}(f)\neq \mathbb{Z}\oplus\mathbb{Z}$, contradicting Step 3. A similar argument shows that $f$ does not divide $f_2$ or $f_3$. Suppose now that $f$ divides $f_i$, for $i\in \{4,5,6,7\}$. Applying the ring automorphisms $x\to -x$, $y\to -y$ to $\mathbb{Q}[n,x,y,x^{-1},y^{-1}]$ shows that $f_1,f_2$, and $f_3$ also divide $f_i$. As $f$,$f_1$,$f_2$, and $f_3$ are coprime, we are led to the product $f\cdot f_1\cdot f_2\cdot f_3$ dividing $f_i$, which is clearly a contradiction in degree. 
\section{Metallic Means and Polygon Diagonals}
We now apply the methods of the previous section to prove Theorem 1.1. Our calculations may be checked in the Sage code provided in the ancillary arXiv file. 
\begin{proof}[Proof of Theorem 1.1]
Suppose $\phi_{n_0}$ is a metallic mean which may be represented as a ratio of diagonals of a regular $N$-gon. By viewing the polygon as the set of points $\{\zeta_N^a:1\leq a\leq N\}$ 
in the complex plane, we obtain 
$$
\phi_{n_0}=\frac{|1-\zeta_N^a|}{|1-\zeta_N^b|}.
$$
Noting that $\overline{\zeta_N}=\zeta_N^{-1}$ and squaring both sides, we obtain 
$$
\phi_{n_0}^{2}=\frac{(1-\zeta_N^a)(1-\zeta_N^{-a})}{(1-\zeta_N^b)(1-\zeta_N^{-b})}.
$$
Now for any metallic ratio $\phi_n$ we have $\phi_n^2-2+\phi_n^{-2}=(\phi_n-\phi_n^{-1})^2=n^2$. In particular, we obtain
$$
\phi_n^2+\phi_n^{-2}-n^2-2=0.
$$
Thus via substitution we see that $(n_0^2,\zeta_N^a,\zeta_N^b)$ must be a zero of the following multivariate polynomial: 
\begin{align*}
f(n,x,y)=-x^3y^3&n^2 + x^4y^2 - 2x^3y^3 + x^2y^4 + 2x^3y^2n^2 + 2x^2y^3n^2 - x^3yn^2 - 4x^2y^2n^2 \\&- xy^3n^2 - 2x^3y + 4x^2y^2 - 2xy^3 + 2x^2yn^2 + 2xy^2n^2 - xyn^2 + x^2 - 2xy + y^2.
\end{align*}

In particular for the classification of Theorem 1.1 we may run the algorithms of Section 2 to find all solutions to $f(n,x,y)=0$ whose first coordinate is rational and last two coordinates are roots of unity. Given such a triple $(a,\tau,\omega)$ we get two desired metallic means, corresponding to the positive and negative square roots of $a$:
\begin{align*}
\phi_{\sqrt{a}}&=\frac{|1-\omega|}{|1-\tau|}\\
\phi_{-\sqrt{a}}&=\frac{|1-\tau|}{|1-\omega|}.
\end{align*}

\end{proof}
Below we list explicit realizations of the metallic means as ratios of polygon diagonals. We omit the trivial solution for $n=0$, and only list solutions for $n$ positive. To obtain the negative solutions one merely needs to take reciprocals.
\begin{center}
\bgroup
\def\arraystretch{1.5}
    \begin{tabular}{|c|c|c|c|}
        \hline
         $n$&$\phi_n$&Numerator Diagonal&Denominator Diagonal\\
         \hline
         $1$&$\frac{1+\sqrt{5}}{2}$&$|1-\zeta_5^2|$&$|1-\zeta_5|$ \\
         \hline
         $2$&${1+\sqrt{2}}$&$|1-\zeta_8^3|$&$|1-\zeta_8|$\\
         \hline
         $\tfrac{3}{2}$&${2}$&$|1-\zeta_6^3|$&$|1-\zeta_6|$\\
         \hline
         $\sqrt{2}$&$\frac{\sqrt{2}+\sqrt{6}}{2}$&$|1-\zeta_{12}^5|$&$|1-\zeta_{12}^2|$\\
         \hline
         $\sqrt{5}$&$\frac{3+\sqrt{5}}{2}$&$|1-\zeta_{10}^3|$&$|1-\zeta_{10}|$\\
         \hline
         $\sqrt{12}$&${2+\sqrt{3}}$&$|1-\zeta_{12}^5|$&$|1-\zeta_{12}|$\\
         \hline
         $\tfrac{\sqrt{2}}{2}$&${\sqrt{2}}$&$|1-\zeta_{24}^6|$&$|1-\zeta_{24}^4|$\\
         \hline
         $\tfrac{\sqrt{6}}{6}$&$\frac{\sqrt{6}}{2}$&$|1-\zeta_{12}^4|$&$|1-\zeta_{12}^3|$\\
         \hline
         $\tfrac{2\sqrt{3}}{3}$&${\sqrt{3}}$&$|1-\zeta_6^2|$&$|1-\zeta_6|$\\
         \hline
         $\tfrac{\sqrt{12}}{12}$&$\frac{2\sqrt{3}}{3}$&$|1-\zeta_6^3|$&$|1-\zeta_6^2|$\\
         \hline
         
    \end{tabular}
\egroup
\end{center}
\section{Generating real subfields of cyclotomic fields}
We now change our focus to the question of when a ratio of polygon diagonals generates a ``small'' extension of $\Q$. To simplify our discussion, we will restrict to the case in which $d_1 = |1-\zeta_N^a|$ and $d_2 = |1-\zeta_N^b|$ with $a$ and $b$ coprime; in doing so we rule out the possibility of, for instance, treating a ratio of square diagonals as a ratio of octagon diagonals, which simplifies the exposition greatly.

We will start by analyzing the complementary question of when a ratio $\tfrac{d_1}{d_2}$ fails to generate the extension $\Q[d_1,d_2]$. From our assumption on $d_1$ and $d_2$ it is straightforward to show $\Q[d_1,d_2] = \Q[\zeta_{4N}^+]$, and thus the degree of $\Q[d_1,d_2]$ over $\Q$ is $\frac{\phi(4N)}{2}$. Since this quantity diverges with $N$, we have reduced Theorem 1.2 to showing that $[\Q[d_1,d_2]:\Q[\tfrac{d_1}{d_2}]]$ is bounded above. As with the previous result, our approach will involve finding cyclotomic solutions to a (Laurent) polynomial. 

Let us call a ratio of polygon diagonals $\tfrac{d_1}{d_2}$ for which $\Q[\tfrac{d_1}{d_2}]\neq \Q[d_1,d_2]$ \emph{defective}. In the same vein as the previous section, write 
\begin{align*}
d_1&=|1-\zeta_N^a|\\
&=\sqrt{(1-\zeta_N^a)(1-\zeta_N^{-a})}\\
&=\sqrt{-\zeta_N^{-a}(1-\zeta_N^a)^2}\\
&=\zeta_{4N}^N\zeta_{2N}^{-a}(1-\zeta_N^a)\\
&=\zeta_{4N}^{N-2a}+\zeta_{4N}^{3N+2a}.
\end{align*}
We note in passing that the product of the two summands above equals 1; this will be crucial in what follows.
We now know that a ratio of two diagonals in a regular $N$-gon may be written as 
$$
\frac{d_1}{d_2}=\frac{\zeta_{4N}^{N-2a}+\zeta_{4N}^{3N+2a}}{\zeta_{4N}^{N-2b}+\zeta_{4N}^{3N+2b}}.
$$
Suppose that $\frac{d_1}{d_2}$ is defective, and let $\sigma$ be a nontrivial automorphism in $\Gal(\mathbb{Q}[d_1,d_2]/\mathbb{Q})$ fixing the ratio. Let $\hat\sigma$ be a lift of $\sigma$ to $\Gal(\Q[\zeta_{4N}]/\Q)$. We have 
\begin{align*}
\frac{\zeta_{4N}^{N-2a}+\zeta_{4N}^{3N+2a}}{\zeta_{4N}^{N-2b}+\zeta_{4N}^{3N+2b}}&=\hat\sigma\left(\frac{\zeta_{4N}^{N-2a}+\zeta_{4N}^{3N+2a}}{\zeta_{4N}^{N-2b}+\zeta_{4N}^{3N+2b}}\right)\\
&=\frac{\hat\sigma(\zeta_{4N}^{N-2a})+\hat\sigma(\zeta_{4N}^{3N+2a})}{\hat\sigma(\zeta_{4N}^{N-2b})+\hat\sigma(\zeta_{4N}^{3N+2b})}.
\end{align*}
We are naturally led to search for cyclotomic solutions to the equation 
$$
\frac{x_1+x_1^{-1}}{x_2+x_2^{-1}}=\frac{y_1+y_1^{-1}}{y_2+y_2^{-1}}.
$$
A simple algebraic manipulation reduces this to finding cyclotomic solutions to 
$$
f(x_1,x_2,y_1,y_2)=(x_1y_2+x_1^{-1}y_2^{-1})+(x_1y_2^{-1}+x_1y_2)+(x_2y_1+x_2^{-1}y_1^{-1})+(x_2y_1^{-1}+x_2^{-1}y_1).
$$
One might hope to find the cyclotomic solutions to $f$ by iteratively taking resultants, as in the previous section. However it turns out that this is computationally infeasible; here is a heuristic for why. To do so would require first scaling by $x_1x_2y_1y_2$ to clear denominators, resulting in a quartic polynomial. Upon taking resultants to reduce to a univariate polynomial, one obtains resultants $r_1(x_2,y_1,y_2)$, $r_2(y_1,y_2)$, and $r_3(y_2)$ of total degrees approximately $20,70,$ and $1300$ respectively. Should we find a cyclotomic zero $\omega$ for $r_3$, one then has to lift this to find a cyclotomic zero to $r_2(y_1,\omega)$. To find the cyclotomic zeroes of this polynomial one first computes a field norm in order to work with a polynomial with purely rational coefficients; but if $\omega$ has conductor $N$, such a norm would have degree $\sim40^{\phi(N)}$. This quickly becomes too much to ask of our computers! Fortunately we can work around this as follows.

Note that $f$ is comprised of eight monomials which appear in complex conjugate pairs when the variables are specialized to roots of unity (hence the parentheses).  Thus while the computational techniques of the previous sections fail, we can still solve the problem by using the techniques of W\l odarski \cite{Wlo69} and of Conway and Jones [CJ] for solving trigonometric diophantine equations. Namely, the identity 
$$2\cos(\tfrac{2\pi a}{b})=\zeta_b^a+\zeta_b^{-a}$$
suggests first finding rational solutions to the equation 
$$
F(A,B,C,D)=\cos(\pi A)+\cos(\pi B)+\cos(\pi C)+\cos(\pi D)=0.
$$
This is in part the content of \cite{Wlo69} and \cite{CJ76}. 
\begin{lemma}[\cite{Wlo69} Theorem 1, \cite{CJ76} Theorem 6]
The rational solutions to 
$$
\cos(\pi A)+\cos(\pi B)+\cos(\pi C)+\cos(\pi D)=0
$$
come in two parametric families and 10 `sporadic' solutions. The parametric families have one of the two following forms:
$$
\{A,B,C,D\}=\{\alpha,\beta,1-\alpha,1-\beta\}
$$
$$
\{A,B,C,D\}=\left\{\alpha,\frac{2}{3}-\alpha,\frac{2}{3}+\alpha,\frac{1}{2}\right\}.
$$
The ten sporadic solutions occur in 5 pairs which are negations of one another; in these cases, $\{A,B,C,D\}$ is one of the following:

\begin{align*}
&\left\{\frac{2}{5},\frac{1}{2},\frac{4}{5},\frac{1}{3}\right\}&\;&\left\{\frac{3}{5},\frac{1}{2},\frac{1}{5},\frac{2}{3}\right\}\\
&\left\{1,\frac{1}{5},\frac{3}{5},\frac{1}{3}\right\}&\;&\left\{0,\frac{4}{5},\frac{2}{5},\frac{2}{3}\right\}\\
&\left\{\frac{2}{5},\frac{7}{15},\frac{13}{15},\frac{1}{3}\right\}&\;&\left\{\frac{3}{5},\frac{8}{15},\frac{2}{15},\frac{2}{3}\right\}\\
&\left\{\frac{4}{5},\frac{1}{15},\frac{11}{15},\frac{1}{3}\right\}&\;&\left\{\frac{1}{5},\frac{14}{15},\frac{4}{15},\frac{2}{3}\right\}\\
&\left\{\frac{2}{7},\frac{4}{7},\frac{6}{7},\frac{1}{3}\right\}&\;&\left\{\frac{1}{7},\frac{3}{7},\frac{5}{7},\frac{2}{3}\right\}\\
\end{align*}

\qed
\end{lemma}
This allows us to find the cyclotomic solutions of our Laurent polynomial in question as follows. Note that not every solution will be germane to the question of when the ratio of two diagonals is defective, since not every solution will be related by an automorphism in $\Gal(\overline{\Q}/\Q)$. However for completeness we will not restrict ourselves by this condition just yet. 

Given a solution $f(\omega_1,\omega_2,\tau_1,\tau_2)=0$, where $\omega_i$ and $\tau_i$ are roots of unity, we are led to a solution of the trigonometric diophantine equation above as discussed. Write the solution as 
$$
F(2\tfrac{a_1}{b_1},2\tfrac{a_2}{b_2},2\tfrac{a_3}{b_3},2\tfrac{a_4}{b_4})=0,
$$
for integers $a_i,b_i$.
This can be rewritten as 
$$
\zeta_{b_1}^{a_1}+\zeta_{b_1}^{-a_1}+\zeta_{b_1}^{a_2}+\zeta_{b_2}^{-a_2}+\zeta_{b_3}^{a_3}+\zeta_{b_3}^{-a_3}+\zeta_{b_4}^{a_4}+\zeta_{b_4}^{-a_4}=0.
$$
It follows that for some permutation $\pi:\{1,2,3,4\}\to \{1,2,3,4\}$,
\begin{align*}
\omega_1\tau_2+\omega_1^{-1}\tau_2^{-1}&=\zeta_{b_{\pi(1)}}^{a_{\pi(1)}}+\zeta_{b_{\pi(1)}}^{-a_{\pi(1)}}\\
\omega_1\tau_2^{-1}+\omega_1^{-1}\tau_2&=\zeta_{b_{\pi(2)}}^{a_{\pi(2)}}+\zeta_{b_{\pi(2)}}^{-a_{\pi(2)}}\\
\omega_2\tau_1+\omega_1^{-1}\tau_2^{-1}&=\zeta_{b_{\pi(3)}}^{a_{\pi(3)}}+\zeta_{b_{\pi(3)}}^{-a_{\pi(3)}}\\
\omega_2\tau_1^{-1}+\omega_2^{-1}\tau_2&=\zeta_{b_{\pi(4)}}^{a_{\pi(4)}}+\zeta_{b_{\pi(4)}}^{-a_{\pi(4)}}.
\end{align*}
To simplify this further, note this implies there is a function $\sign:\{1,2,3,4\}\to \{1,-1\}$ for which 
\begin{align*}
\omega_1\tau_2&=\zeta_{b_{\pi(1)}}^{\sign(1)a_{\pi(1)}}\\
\omega_1\tau_2^{-1}&=\zeta_{b_{\pi(2)}}^{\sign(2)a_{\pi(2)}}\\
\omega_2\tau_1&=\zeta_{b_{\pi(3)}}^{\sign(3)a_{\pi(3)}}\\
\omega_2\tau_1^{-1}&=\zeta_{b_{\pi(4)}}^{\sign(4)a_{\pi(4)}}.
\end{align*}
Once we have reduced to these 4 equalities it is simple to solve for $\omega_1,\omega_2,\tau_1,\tau_2$ up to negation using, for instance,
\begin{align*}
\omega_1^2&=\zeta_{b_{\pi(1)}}^{\sign(1)a_{\pi(1)}}\cdot \zeta_{b_{\pi(2)}}^{\sign(2)a_{\pi(2)}}\\
\omega_2^2&=\zeta_{b_{\pi(3)}}^{\sign(3)a_{\pi(3)}}\cdot \zeta_{b_{\pi(4)}}^{\sign(4)a_{\pi(4)}}.
\end{align*}
As we are solving a quadratic equation for $\omega_i$, each choice of $\pi$ and $\sign$ will lead to $4$ solutions. Thus a priori, for each solution 
$$
F(2\tfrac{a_1}{b_1},2\tfrac{a_2}{b_2},2\tfrac{a_3}{b_3},2\tfrac{a_4}{b_4})=0
$$
we should expect $4!\cdot 2^4\cdot 4=1,536$ solutions $f(\omega_1,\omega_2,\tau_1,\tau_2)=0$. For the sake of brevity in this paper we will reduce this in three ways. Note that $f$ admits the following symmetries: 
\begin{enumerate}
\item $f(x_1,x_2,y_1,y_2)=f(x_2,x_1,y_2,y_1)$\\
\item $f(x_1,x_2,y_1,y_2)=f(x_1^{-1},x_2,y_1,y_2)$\\
\item $f(x_1,x_2,y_1,y_2)=f(x_1,x_2^{-1},y_1,y_2)$\\
\item $f(x_1,x_2,y_1,y_2)=f(y_2,x_2,y_1,x_1)$\\
\item $f(x_1,x_2,y_1,y_2)=f(x_1,y_1,x_2,y_2)$\\
\item $f(x_1,x_2,y_1,y_2)=f(y_2^{-1},x_2,y_1,x_1^{-1})$\\
\item $f(x_1,x_2,y_1,y_2)=f(x_1,y_1^{-1},x_2^{-1},y_2)$.\\
\end{enumerate}
The first three operations allow us to reduce the number of permutations that we need to check; for instance, if we have a permutation assigning 
\begin{align*}
\omega_1\tau_2+\omega_1^{-1}\tau_2^{-1}&=\zeta_{b_{1}}^{a_{1}}+\zeta_{b_{1}}^{-a_{1}}\\
\omega_1\tau_2^{-1}+\omega_1^{-1}\tau_2&=\zeta_{b_{2}}^{a_{2}}+\zeta_{b_{2}}^{-a_{2}}\\
\omega_2\tau_1+\omega_1^{-1}\tau_2^{-1}&=\zeta_{b_{3}}^{a_{3}}+\zeta_{b_{3}}^{-a_{3}}\\
\omega_2\tau_1^{-1}+\omega_2^{-1}\tau_2&=\zeta_{b_{4}}^{a_{4}}+\zeta_{b_{4}}^{-a_{4}},
\end{align*}
then applying the first listed symmetry gives a solution for which 
\begin{align*}
\omega_1\tau_2+\omega_1^{-1}\tau_2^{-1}&=\zeta_{b_{3}}^{a_{3}}+\zeta_{b_{3}}^{-a_{3}}\\
\omega_1\tau_2^{-1}+\omega_1^{-1}\tau_2&=\zeta_{b_{4}}^{a_{4}}+\zeta_{b_{4}}^{-a_{4}}\\
\omega_2\tau_1+\omega_1^{-1}\tau_2^{-1}&=\zeta_{b_{1}}^{a_{1}}+\zeta_{b_{1}}^{-a_{1}}\\
\omega_2\tau_1^{-1}+\omega_2^{-1}\tau_2&=\zeta_{b_{2}}^{a_{2}}+\zeta_{b_{2}}^{-a_{2}}.
\end{align*}
Thus by composing operations 1 thru 3 we can generate eight distinct permutations, and hence we will only list three representatives explicitly which generate all 24 under these operations. Similarly, operations 4 thru 7 allow us to change the sign function arbitrarily; again, by way of example if we have a solution for which
\begin{align*}
\omega_1\tau_2&=\zeta_{b_{1}}^{a_{1}}\\
\omega_1\tau_2^{-1}&=\zeta_{b_{2}}^{a_{2}}\\
\omega_2\tau_1&=\zeta_{b_{3}}^{a_{3}}\\
\omega_2\tau_1^{-1}&=\zeta_{b_{4}}^{a_{4}}
\end{align*}
then after applying operation 4 we will have a solution for which 
\begin{align*}
\omega_1\tau_2&=\zeta_{b_{1}}^{a_{1}}\\
\omega_1\tau_2^{-1}&=\zeta_{b_{2}}^{-a_{2}}\\
\omega_2\tau_1&=\zeta_{b_{3}}^{a_{3}}\\
\omega_2\tau_1^{-1}&=\zeta_{b_{4}}^{a_{4}}.
\end{align*}
Thus for a fixed permutation there is no need to worry about a sign function as long as we allow for arbitrary composition of operations 4 thru 7. Finally, we will rampantly use the plus or minus sign $\pm$ to simplify the solutions to quadratic equations. We hope no confusion arises from these simplifications.

With this discussion in mind, we have the following classification. The proof is immediate given Lemma 4.1 and the previous discussion.
\begin{lemma}
Let $(\omega_1,\omega_2,\tau_1,\tau_2)$ be a cyclotomic solution to $f$. Then $(\omega_1,\omega_2,\tau_1,\tau_2)$ can be obtained via the seven listed symmetries from one of the following families of solutions:
\begin{itemize}
\item For any root of unity $\zeta_b^a$, one of the following three parametric families:
\begin{enumerate}
\item 
$(x_1,y_2) = \pm(\zeta_3^2\zeta_b^a,\zeta_3)$,
$(x_2,y_1) = \pm(\zeta_{24}^{11}\zeta_{2b}^a,\zeta_{24}^5\zeta_{2b}^a)$.
\item 
$(x_1,y_2) = \pm(\zeta_3\zeta_b^a,\zeta_3^2)$,
$(x_2,y_1) = \pm(\zeta_{24}^7\zeta_{2b}^a,\zeta_{24}\zeta_{2b}^a)$.
\item 
$(x_1,y_2) = \pm(\zeta_8\zeta_{2b}^a,\zeta_8^7\zeta_{2b}^a)$,
$(x_2,y_1) = \pm(\zeta_b^a,\zeta_3)$.
\end{enumerate}
\item For any two roots of unity $\zeta_b^a, \zeta_c^d$, one of the following three parametric families:
\begin{enumerate}
\item 
$(x_1,y_2) = \pm(\zeta_4\zeta_b^a,\zeta_4^3)$,
$(x_2,y_1) = \pm(\zeta_4,\zeta_c^d,\zeta_4^3)$.
\item
$(x_1,y_2) = \pm(\zeta_{2bd}^{ad+bc},\zeta_{2bd}^{ad-bc})$,
$(x_2,y_1) = \pm(\zeta_{2bd}^{-ad-bc},\zeta_{2bd}^{-ad+bc})$.
\item 
$(x_1,y_2) = \pm(\zeta_4\zeta_{2bd}^{ad+bc},\zeta_4^3\zeta_{2bd}^{ad-bc})$,
$(x_2,y_1) = \pm(\zeta_4\zeta_{2bd}^{-ad-bc},\zeta_4^3\zeta_{2bd}^{-ad+bc})$.
\end{enumerate}
\item Any entry from the following table:
\end{itemize}
$$
\begin{array}{|c|c|c|c|c|c|}
\hline
x_1y_2+x_1^{-1}y_2^{-1}&x_1y_2^{-1}+x_1^{-1}y_2&x_2y_1+x_2^{-1}y_1^{-1}&x_2y_1^{-1}+x_2^{-1}y_1&(x_1,y_2)&(x_2,y_1)\\
\hline
\hline
\zeta_{5}+\zeta_5^{4}&\zeta_4+\zeta_4^{3}&\zeta_5^2+\zeta_5^{3}&\zeta_6+\zeta_6^{5}&\pm(\zeta_{40}^9,\zeta_{40}^{39})&\pm(\zeta_{60}^{17},\zeta_{60}^{7})\\
\hline
\zeta_{5}+\zeta_5^{4}&\zeta_5^2+\zeta_5^{3}&\zeta_4+\zeta_4^{3}&\zeta_6+\zeta_6^{5}&\pm(\zeta_{10}^3,\zeta_{10}^{9})&\pm(\zeta_{24}^{5},\zeta_{24})\\
\hline
\zeta_{5}+\zeta_5^{4}&\zeta_6+\zeta_6^{5}&\zeta_4+\zeta_4^{3}&\zeta_5^2+\zeta_5^{3}&\pm(\zeta_{60}^{11},\zeta_{60})&\pm(\zeta_{40}^{13},\zeta_{40}^{37})\\
\hline
\hline
\zeta_{10}^3+\zeta_{10}^{7}&\zeta_4+\zeta_4^{3}&\zeta_{10}+\zeta_{10}^9&\zeta_3+\zeta_3^{2}&\pm(\zeta_{40}^{11},\zeta_{40})&\pm(\zeta_{60}^{13},\zeta_{60}^{53})\\
\hline
\zeta_{10}^3+\zeta_{10}^{7}&\zeta_{10}+\zeta_{10}^9&\zeta_4+\zeta_4^{3}&\zeta_3+\zeta_3^{2}&\pm(\zeta_{10}^2,\zeta_{10})&\pm(\zeta_{24}^{7},\zeta_{24}^{23})\\
\hline
\zeta_{10}^3+\zeta_{10}^{7}&\zeta_3+\zeta_3^{2}&\zeta_4+\zeta_4^{3}&\zeta_{10}+\zeta_{10}^9&\pm(\zeta_{60}^{19},\zeta_{60}^{59})&\pm(\zeta_{40}^{7},\zeta_{40}^{3})\\
\hline
\hline
\zeta_{2}+\zeta_2&\zeta_{10}+\zeta_{10}^{9}&\zeta_{10}^3+\zeta_{10}^{7}&\zeta_6+\zeta_6^{5}&\pm(\zeta_{10}^3,\zeta_{10}^{2})&\pm(\zeta_{60}^{14},\zeta_{60}^{4})\\
\hline
\zeta_{2}+\zeta_2&\zeta_{10}^3+\zeta_{10}^{7}&\zeta_{10}+\zeta_{10}^{9}&\zeta_6+\zeta_6^{5}&\pm(\zeta_{10}^4,\zeta_{10})&\pm(\zeta_{60}^{8},\zeta_{60}^{58})\\
\hline
\zeta_{2}+\zeta_2&\zeta_6+\zeta_6^{5}&\zeta_{10}+\zeta_{10}^{9}&\zeta_{10}^3+\zeta_{10}^{7}&\pm(\zeta_{6}^2,\zeta_{6})&\pm(\zeta_{10}^{2},\zeta_{10}^{9})\\
\hline
\hline
\zeta_{2}^0+\zeta_{2}^{0}&\zeta_5^2+\zeta_5^{3}&\zeta_{5}+\zeta_{5}^4&\zeta_3+\zeta_3^{2}&\pm(\zeta_{5},\zeta_{5}^4)&\pm(\zeta_{15}^{4},\zeta_{15}^{14})\\
\hline
\zeta_{2}^0+\zeta_{2}^{0}&\zeta_{5}+\zeta_{5}^4&\zeta_5^2+\zeta_5^{3}&\zeta_3+\zeta_3^{2}&\pm(\zeta_{10},\zeta_{10}^9)&\pm(\zeta_{30}^{11},\zeta_{30})\\
\hline
\zeta_{2}^0+\zeta_{2}^{0}&\zeta_3+\zeta_3^{2}&\zeta_5^2+\zeta_5^{3}&\zeta_{5}+\zeta_{5}^4&\pm(\zeta_{6},\zeta_{6}^5)&\pm(\zeta_{10}^3,\zeta_{10})\\
\hline
\hline
\zeta_{5}+\zeta_5^4&\zeta_{30}^7+\zeta_{30}^{23}&\zeta_{30}^{13}+\zeta_{30}^{17}&\zeta_6+\zeta_6^{5}&\pm(\zeta_{60}^{13},\zeta_{60}^{59})&\pm(\zeta_{30}^{9},\zeta_{30}^{4})\\
\hline
\zeta_{5}+\zeta_5^4&\zeta_{30}^{13}+\zeta_{30}^{17}&\zeta_{30}^7+\zeta_{30}^{23}&\zeta_6+\zeta_6^{5}&\pm(\zeta_{60}^{19},\zeta_{60}^{53})&\pm(\zeta_{30}^{6},\zeta_{30})\\
\hline
\zeta_{5}+\zeta_5^4&\zeta_6+\zeta_6^{5}&\zeta_{30}^7+\zeta_{30}^{23}&\zeta_{30}^{13}+\zeta_{30}^{17}&\pm(\zeta_{60}^{11},\zeta_{60})&\pm(\zeta_{30}^{10},\zeta_{30}^{27})\\
\hline
\hline
\zeta_{10}^3+\zeta_{10}^{7}&\zeta_{15}^4+\zeta_{15}^{11}&\zeta_{15}+\zeta_{15}^{14}&\zeta_3+\zeta_3^{2}&\pm(\zeta_{60}^{17},\zeta_{60}^{59})&\pm(\zeta_{30}^{6},\zeta_{30}^{26})\\
\hline
\zeta_{10}^3+\zeta_{10}^{7}&\zeta_{15}+\zeta_{15}^{14}&\zeta_{15}^4+\zeta_{15}^{11}&\zeta_3+\zeta_3^{2}&\pm(\zeta_{60}^{11},\zeta_{60}^{7})&\pm(\zeta_{30}^{9},\zeta_{30}^{29})\\
\hline
\zeta_{10}^3+\zeta_{10}^{7}&\zeta_3+\zeta_3^{2}&\zeta_{15}^4+\zeta_{15}^{11}&\zeta_{15}+\zeta_{15}^{14}&\pm(\zeta_{60}^{19},\zeta_{60}^{59})&\pm(\zeta_{30}^{5},\zeta_{30}^{3})\\
\hline
\hline
\zeta_{5}^2+\zeta_5^3&\zeta_{30}+\zeta_{30}^{29}&\zeta_{30}^{11}+\zeta_{30}^{19}&\zeta_6+\zeta_6^{5}&\pm(\zeta_{60}^{13},\zeta_{60}^{11})&\pm(\zeta_{30}^{8},\zeta_{30}^{3})\\
\hline
\zeta_{5}^2+\zeta_5^3&\zeta_{30}^{11}+\zeta_{30}^{19}&\zeta_{30}+\zeta_{30}^{29}&\zeta_6+\zeta_6^{5}&\pm(\zeta_{60}^{23},\zeta_{60})&\pm(\zeta_{30}^{3},\zeta_{30}^{28})\\
\hline
\zeta_{5}^2+\zeta_5^3&\zeta_6+\zeta_6^{5}&\zeta_{30}+\zeta_{30}^{29}&\zeta_{30}^{11}+\zeta_{30}^{19}&\pm(\zeta_{60}^{17},\zeta_{60}^{7})&\pm(\zeta_{30}^{6},\zeta_{30}^{25})\\
\hline
\hline
\zeta_{10}+\zeta_{10}^{9}&\zeta_{15}^7+\zeta_{15}^{8}&\zeta_{15}^2+\zeta_{15}^{13}&\zeta_3+\zeta_3^{2}&\pm(\zeta_{60}^{17},\zeta_{60}^{11})&\pm(\zeta_{30}^{7},\zeta_{30}^{27})\\
\hline
\zeta_{10}+\zeta_{10}^{9}&\zeta_{15}^2+\zeta_{15}^{13}&\zeta_{15}^7+\zeta_{15}^{8}&\zeta_3+\zeta_3^{2}&\pm(\zeta_{60}^{7},\zeta_{60}^{59})&\pm(\zeta_{30}^{12},\zeta_{30}^{2})\\
\hline
\zeta_{10}+\zeta_{10}^{9}&\zeta_3+\zeta_3^{2}&\zeta_{15}^7+\zeta_{15}^{8}&\zeta_{15}^2+\zeta_{15}^{13}&\pm(\zeta_{60}^{13},\zeta_{60}^{53})&\pm(\zeta_{30}^{9},\zeta_{30}^{5})\\
\hline
\hline
\zeta_{7}+\zeta_7^6&\zeta_{7}^2+\zeta_{7}^{5}&\zeta_{7}^{3}+\zeta_{7}^{4}&\zeta_6+\zeta_6^{5}&\pm(\zeta_{42}^{9},\zeta_{42}^{39})&\pm(\zeta_{84}^{25},\zeta_{84}^{11})\\
\hline
\zeta_{7}+\zeta_7^6&\zeta_{7}^{3}+\zeta_{7}^{4}&\zeta_{7}^2+\zeta_{7}^{5}&\zeta_6+\zeta_6^{5}&\pm(\zeta_{42}^{12},\zeta_{42}^{36})&\pm(\zeta_{84}^{19},\zeta_{84}^{5})\\
\hline
\zeta_{7}+\zeta_7^6&\zeta_6+\zeta_6^{5}&\zeta_{7}^2+\zeta_{7}^{5}&\zeta_{7}^{3}+\zeta_{7}^{4}&\pm(\zeta_{84}^{13},\zeta_{84}^{83})&\pm(\zeta_{42}^{15},\zeta_{42}^{39})\\
\hline
\hline
\zeta_{14}+\zeta_{14}^{13}&\zeta_{14}^3+\zeta_{14}^{11}&\zeta_{14}^{5}+\zeta_{14}^{9}&\zeta_3+\zeta_3^{2}&\pm(\zeta_{42}^{6},\zeta_{42}^{39})&\pm(\zeta_{84}^{29},\zeta_{84})\\
\hline
\zeta_{14}+\zeta_{14}^{13}&\zeta_{14}^{5}+\zeta_{14}^{9}&\zeta_{14}^3+\zeta_{14}^{11}&\zeta_3+\zeta_3^{2}&\pm(\zeta_{42}^{9},\zeta_{42}^{36})&\pm(\zeta_{84}^{23},\zeta_{84}^{79})\\
\hline
\zeta_{14}+\zeta_{14}^{13}&\zeta_3+\zeta_3^{2}&\zeta_{14}^3+\zeta_{14}^{11}&\zeta_{14}^{5}+\zeta_{14}^{9}&\pm(\zeta_{84}^{17},\zeta_{84}^{73})&\pm(\zeta_{42}^{12},\zeta_{42}^{39})\\
\hline
\end{array}
$$
\qed
\end{lemma}
We can now refine the previous lemma by taking into consideration Galois groups of real cyclotomic extensions.
\begin{lemma}
Take two roots of unity, written as $\zeta_{2N}^a$ and $\zeta_{2N}^b$. Then there exists a nontrivial automorphism $\sigma\in \Gal({\Q[\zeta_{2N}]}/\Q)$ for which 
$$
\frac{\zeta_{2N}^a+\zeta_{2N}^{-a}}{\zeta_{2N}^b+\zeta_{2N}^{-b}}=
\frac{\sigma(\zeta_{2N}^a)+\sigma(\zeta_{2N}^{-a})}{\sigma(\zeta_{2N}^b)+\sigma(\zeta_{2N}^{-b})}
$$
if and only if either 
\begin{itemize}
\item $\zeta_{2N}^a= \pm \zeta_{2N}^{\pm b},$ or \\
\item there exists a $k\in (\mathbb{Z}/2N\mathbb{Z})^\times$ which solves the simultaneous equations
\begin{align*}
ak&\equiv N+a\mod 2N\\
bk&\equiv N+b\mod 2N
\end{align*}
\end{itemize}
\end{lemma}
Before giving the proof which, at this stage, follows quite simply from our buildup, we provide an example application of our theorem. Let us take $2N=20$, $a=1$, and $b=3$. Then 
\begin{align*}
k&\equiv 10+1 \mod 20\\
3k&\equiv 10+3 \mod 20
\end{align*}
has the solution $k=11$, telling us that $\frac{\zeta_{20}+\zeta_{20}^{19}}{\zeta_{20}^3+\zeta_{20}^{17}}$ fails to generate $\Q[\zeta_{20}+\zeta_{20}^{-1}]$. However, the ratio $\frac{\zeta_{20}+\zeta_{20}^{19}}{\zeta_{20}^2+\zeta_{20}^{18}}$ does generate the extension, since if $k$ is forced to be $11$ then $2k\not\equiv 12\mod 20$. In this simple case one could verify these observations by, for instance, computing the minimal polynomials of the previous two ratios. However, one could imagine that if $N$ grows it will become much simpler to use the criterion given by Lemma 4.3.

We now proceed with the proof of Lemma 4.3.
\begin{proof}[Proof of Lemma 4.3]
Suppose $\sigma$ is a map as above. As a field automorphism, $\sigma$ must respect multiplicative order when applied to any root of unity. By examining Lemma 4.2 we see that the only case in which this occurs is if $x_1=\pm x_2^{\pm1}$, in which case $\sigma$ can be arbitrary, or if we can simultaneously solve $y_1=\pm x_1^{\pm 1}$ and $y_2=\pm x_2^{\pm1}$. Thus we must find an automorphism $\sigma$ for which $\sigma(\zeta_{2N}^a)=\pm \zeta_{2N}^{\pm a}$ and $\sigma(\zeta_{2N}^b)=\pm \zeta_{2N}^{\pm b}$.
Now automorphisms of cyclotomic extensions will have the form $\zeta_{2N}\to \zeta_{2N}^k$ for $k$ coprime to $2N$. In order for the automorphism to be nontrivial, we must be able to solve the following consistently for $k$:
\begin{align*}
ka&\equiv N+a \mod 2N\\
kb&\equiv N+b \mod 2N.
\end{align*}
This completes the proof.
\end{proof}
Having classified the cyclotomic solutions to our Laurent polynomial of interest and then examined the relationship between these solutions and Galois groups of real cyclotomic fields, it is a simple matter to classify defective ratios of polygon diagonals. 
\begin{proposition}
Let $d_1$ and $d_2$ be diagonals of a regular $N$ gon, with $d_1=|1-\zeta_N^a|$ and $d_2=|1-\zeta_N^b|$. Then the ratio $\tfrac{d_1}{d_2}$ is defective if and only if either of the following hold:
\begin{itemize}
\item $d_1=d_2$,\text{ or }
\item there exists $k\in (\Z/4N\Z)^\times$ which solves the simultaneous equation 
\begin{align*}
k(N-2a)&\equiv 3N-2a \mod 4N\\
k(N-2b)&\equiv 3N-2b \mod 4N.
\end{align*}
\end{itemize}
\end{proposition}
\begin{proof}
As discussed in the beginning of the section, if $\tfrac{d_1}{d_2}$ fails to generate the extension in question then we could find a nontrivial field automorphism $\sigma$ of $\mathbb{Q}[d_1,d_2]$ fixing $\tfrac{d_1}{d_2}$. Writing 
$$
\frac{d_1}{d_2}=
\frac{\zeta_{4N}^{N-2a}+\zeta_{4N}^{3N+2a}}{\zeta_{4N}^{N-2b}+\zeta_{4N}^{3N+2b}},
$$
this would give a solution 
$$
f(\zeta_{4N}^{N-2a},\zeta_{4N}^{N-2b},\sigma(\zeta_{4N}^{N-2a}),\sigma(\zeta_{4N}^{N-2b})=0.
$$
If $d_1\neq d_2$ then Lemma 4.2 implies that $\sigma(\zeta_{4N}^{N-2a})$ is in $\{\pm\zeta_{4N}^{N-2a},\pm\zeta_{4N}^{2a-N}\}$ and that $\sigma(\zeta_{4N}^{N-2b})$ is in $\{\pm\zeta_{4N}^{N-2b},\pm\zeta_{4N}^{2b-N}\}$. As we are working in $\Q[\zeta_{4N}]^+$, we can assume with no loss of generality that $\sigma(\zeta_{4N}^{N-2a})=-\zeta_{4N}^{N-2a}$ and $\sigma(\zeta_{4N}^{N-2b})=-\zeta_{4N}^{N-2b}$. This then implies that there exists a $k\in (\mathbb{Z}/4N\mathbb{Z})^\times$ for which $k(N-2a)\equiv 3N-2a\mod 4N$ and $k(N-2b)\equiv 3N-2b\mod 4N$ as in Lemma 4.3.
\end{proof}
Again by way of example, the golden ratio $\phi_1=\frac{|1-\zeta_5^2|}{|1-\zeta_5|}$ is defective; this reduces to the example following Lemma 4.2, since 
$$
\phi_1=\frac{|1-\zeta_5^2|}{|1-\zeta_5|}=\frac{\zeta_{20}+\zeta_{20}^{19}}{\zeta_{20}^3+\zeta_{20}^{17}}=\frac{\zeta_{20}^{11}+\zeta_{20}^9}{\zeta_{20}^{13}+\zeta_{20}^7}.
$$

An example of a non-defective ratio can also be found by taking, for instance, a ratio of two diagonals in a decagon. For example,
$$
\frac{|1-\zeta_{10}|}{|1-\zeta_{10}^2|}=\frac{\zeta_{40}^8+\zeta_{40}^{32}}{\zeta_{40}^6+\zeta_{40}^6};
$$
since $\zeta_{40}^8$ reduces to a fifth root of unity it cannot be conjugated to its negative, which has multiplicative order $10$. Hence this ratio is not defective. 

We end the section by using Proposition 4.4 to prove Theorem 1.2, and then using Theorem 1.2 to give an alternate proof to Theorem 1.1.
\begin{proof}[Proof of Theorem 1.2]
We can compute the index $[\Q[d_1,d_2]:\Q[\tfrac{d_1}{d_2}]]$ by computing the size of the associated Galois group $\Gal(\Q[d_1,d_2]/\Q[\tfrac{d_1}{d_2}])$. As the second field is monogenic and our extensions are abelian, we may compute this by counting the number of automorphisms $\sigma\in \Gal(\Q[d_1,d_2]/\Q)$ which fix $\tfrac{d_1}{d_2}$. Proposition 4.4 tells us that the number of nontrivial automorphisms which fix $\tfrac{d_1}{d_2}$
equals the one half the number of solutions to the simultaneous equation
\begin{align*}
k(N-2a)&\equiv 3N-2a \mod 4N\\
k(N-2b)&\equiv 3N-2b \mod 4N.
\end{align*}
We must divide by 2 to account for the passing from $\Gal(\Q[\zeta_{4N}]/\Q)$ to $\Gal(\Q[\zeta_{4N}^+]/\Q)$.

Reducing these equation mod $N$ and rearranging shows that such a $k$ must also solve 
\begin{align*}
2(k-1)a &\equiv 0 \mod N\\
2(k-1)b &\equiv 0 \mod N.
\end{align*}
Since $a$ and $b$ are assumed to be coprime, this can only occur if $2(k-1)\equiv 0\mod N$. If $N$ is odd then this equation has the unique solution $k\equiv 1 \mod N$, and hence there are at most 4 solutions to the original solution, corresponding to the residue classes $1$, $N+1$, $2N+1$, and $3N+1 \mod 4N$. A simple parity argument shows $k$ must be odd, and hence there are in fact $2$ solutions in this case. 

Alternatively take $N$ to be even; in this case there are fewer restrictions. The equation $2(k-1)\equiv 0\mod N$ implies $k\equiv 1\mod N$ or $k\equiv N/2+1\mod N$, which lift to at most $8$ solutions to the original equation.

Thus when $N$ is odd there is a unique nontrivial map fixing $\tfrac{d_1}{d_2}$, and when $N$ is even there are at most four such maps. After accounting for the identity morphism of the Galois group we obtain
$$
[\Q[d_1,d_2]:\Q[\tfrac{d_1}{d_2}]] \geq \begin{cases}2\text{ if }N\text{ is odd,}\\5\text{ if }N\text{ is even.}\end{cases}
$$
Using multiplicativity of degree in towers of field extensions then gives Theorem 1.2. 

\end{proof}

We end by showing how Theorem 1.2 gives an alternate proof to Theorem 1.1. 
\begin{proof}[Proof of Theorem 1.1]
The proof of Theorem 1.2 implies that if $d_1$ and $d_2$  are diagonals of a regular $N$-gon, then the ratio $\tfrac{d_1}{d_2}$ generates an extension of $\Q$ of degree at least $\phi(4N)/16$. We employ a crude bound of $\phi(n)\geq 48\cdot(n/210)^{12/13}$ for any integer $n$; for reference, see the Math StackExchange answer \cite{MSE}. In particular, if $\tfrac{d_1}{d_2}$ generates an extension of degree $\leq 4$, then $\phi(4N)\leq 64$, and hence $N\leq 286$. 

To prove Theorem 1.1, all one has to do is iterate through the values $3\leq N\leq 286$. For each such $N$ one computes all possible pairs of coprime integers $1\leq a,b\leq 286$, and for each pair computes the minimal polynomial of
$$
\frac{\zeta_{4N}^{N-2a}+\zeta_{4N}^{3N+2a}}{\zeta_{4N}^{N-2b}+\zeta_{4N}^{3N+2b}}
$$
using a computer algebra system such as Sage \cite{sagemath}. It is then a simple matter to classify which such minimal polynomials give rise to metallic ratios.
\end{proof}

We end our paper by noting that there has recently been much interest in examining vanishing sums of roots of unity. Apart from the work of W{\l}odarski and of Conway and Jones mentioned in this section, one could explore \cite{CMS11}, \cite{LL00}, \cite{PR98}, or \cite{Ste08}. It is feasible to imagine that the methods of this section, combined with these results, could lead to many more answers regarding generators of (subfields of) cyclotomic fields. We leave this area of exploration to the interested reader.

For an alternate avenue of study, one could examine analytical questions regarding polygon diagonals, instead of the algebraic ones examined here. For instance, we propose the following question.
\begin{question}
Let $D$ denote the set of all ratios of polygon diagonals. What are the limit points of $D$?
\end{question}
\noindent For questions of a similar nature regarding Salem or Pisot numbers, see for instance \cite{Smy15} or Chapter 6 of \cite{MR1187044}.
\bibliography{cyc_ref}{}
\bibliographystyle{alpha}

\end{document}